\documentclass[reqno,11pt]{amsart}
\usepackage[letterpaper,hmargin=1.25in,vmargin=0.9in]{geometry}
\usepackage[T1]{fontenc}
\usepackage{lmodern}
\usepackage{microtype}
\usepackage{mathtools}
\usepackage{amssymb}

\usepackage{xcolor}
\usepackage[unicode]{hyperref}
\hypersetup{
    colorlinks=false,
    citecolor=blue,
    linkcolor=blue,
    urlcolor=blue
}
\newcommand{\R}{\mathbb R}
\newcommand{\Z}{\mathbb Z}
\DeclareMathOperator{\conv}{conv}
\DeclareMathOperator{\cs}{cs}
\DeclareMathOperator{\interior}{int}
\DeclareMathOperator{\vol}{vol}
\DeclareMathOperator{\spanop}{span}
\newcommand{\vn}[2]{\vol_{#1}(#2)}
\newcommand{\susp}[1]{\mathcal S(#1)}

\theoremstyle{plain}
\newtheorem{theorem}{Theorem}[section]
\newtheorem{lemma}[theorem]{Lemma}

\newtheorem{corollary}[theorem]{Corollary}

\theoremstyle{remark}
\newtheorem{remark}[theorem]{Remark}

\numberwithin{equation}{section}
\allowdisplaybreaks[1]
\title[Consequences of Mahler's inequalities]{A note on consequences of Mahler's inequalities}
\author{Yanlu Lian}
\address{School of Mathematics, Hangzhou Normal University, Hangzhou, 311121, China}
\email{yllian@hznu.edu.cn}

\author[Fei Xue]{Fei Xue\textsuperscript{*}}
\thanks{\textsuperscript{*}Corresponding author.}
\address{School of Mathematical Sciences, Nanjing Normal University, Nanjing, 210046, China}
\email{05429@njnu.edu.cn}

\thanks{This work is supported by the National Natural Science Foundation of China grant (NSFC12671118), the Natural Science Foundation of Zhejiang
Province grant (LQ24A010008) and the Scientific Research Fund of Zhejiang Provincial Education Department
grant (Y202351675).}
\hypersetup{pdftitle={A note on consequences of Mahler's inequalities},pdfauthor={Yanlu Lian; Fei Xue},pdfsubject={Makai's inequality from symmetric Mahler and basic Minkowski-type consequences},pdfkeywords={Makai's inequality, Mahler's inequality, successive minima, first covering minimum, mixed volume}}

\begin{document}
\begin{abstract}
We derive Makai's inequality $\vol_n(K)\vol_n((K-K)^\circ)\ge(n+1)/n!$ in every dimension from the symmetric Mahler theorem in one higher dimension, with equality precisely for simplices. The proof uses a classical two-layer construction, an identity for its polar volume, and a Rogers--Shephard volume estimate. We then record basic consequences of the Mahler and Makai inequalities for polar successive minima, successive minima at the Santal\'o point, critical determinants, first covering minima, and non-separable lattice arrangements. We also combine Alexandrov--Fenchel with the volume-product inequalities to obtain sharp bounds for paired mixed volumes.
\end{abstract}
\maketitle

\section{Introduction}

A convex body is a compact convex set with nonempty interior. Throughout, $n\ge1$, and $\vol_n$ denotes $n$-dimensional Lebesgue volume. A convex body $C$ is called $o$-symmetric if $C=-C$. For a convex body $K\subset\R^n$ containing the origin in its interior, its polar is
$$K^\circ=\{y\in\R^n:\langle x,y\rangle\le1\text{ for all }x\in K\}.$$
For an arbitrary convex body $K\subset\R^n$, put
$$D(K)=K-K,\qquad \cs(K)=\frac{K-K}{2}.$$
Both bodies are $o$-symmetric; the second is the central symmetral of $K$. Makai's conjecture \cite{Makai,MakaiMartini} asserts that simplices minimize $\vol_n(K)\vol_n(D(K)^\circ)$. This product is invariant under translations and invertible linear transformations.

We consider consequences of the symmetric and general Mahler inequalities; see the manuscripts \cite{OpenAIGeneral,OpenAISymmetric}. Our results are formulated as implications from these inequalities. The symmetric Mahler conjecture asserts that every $o$-symmetric convex body $C\subset\R^n$ satisfies
\begin{equation}\label{eq:sm}
 \vn nC\,\vn n{C^\circ}\geq\frac{4^n}{n!}.
\end{equation}
Applying this theorem in dimension $n+1$ proves Makai's conjecture in dimension $n$.

\begin{theorem}[Makai's inequality]\label{thm:transfer}
Let $n\geq 1$. Every convex body $K\subset\R^n$ satisfies
\begin{equation}\label{eq:makai}
 \vn nK\,\vn n{D(K)^\circ}\geq\frac{n+1}{n!}.
\end{equation}
Equality holds if and only if $K$ is an $n$-simplex.
\end{theorem}

The normalization in \cite[Conjecture~3.17]{MakaiMartini} uses $\cs(K)$; since $\cs(K)^\circ=2D(K)^\circ$, the equivalent constant is $2^n(n+1)/n!$. In dimension three, Aliev \cite[Theorem~1.2]{Aliev} proved \eqref{eq:makai} for convex bodies with tetrahedral symmetry, including the tetrahedron equality case.

The proof uses the $o$-symmetric body $\conv\bigl((K\times\{1\})\cup(-K\times\{-1\})\bigr)$ in $\R^{n+1}$. Its polar volume is $2^{n+1}/(n+1)$ times $\vol_n(D(K)^\circ)$, and its volume is at most $2^{n+1}/(n+1)$ times $\vol_n(K)$. These relations, proved in Lemma~\ref{lem:suspension}, give \eqref{eq:makai} when combined with symmetric Mahler in dimension $n+1$.

The Mahler and Makai inequalities also give basic bounds for polar successive minima, successive minima at the Santal\'o point, first covering minima, and densities of non-separable lattice arrangements. The final section applies Alexandrov-Fenchel to paired mixed volumes and expresses the proof of Theorem~\ref{thm:transfer} in terms of reflected mixed volumes.

\section{From symmetric Mahler to Makai}

For a convex body $K\subset\R^n$, put
$$\susp K=\conv\bigl((K\times\{1\})\cup(-K\times\{-1\})\bigr)\subset\R^{n+1},\qquad a_n=\frac{2^{n+1}}{n+1}.$$
The body $\susp K$ is $o$-symmetric. It is a classical $C$-body\footnote{Here $C$ belongs to the historical name and is unrelated to our notation for an $o$-symmetric convex body.} of Rogers and Shephard \cite{RogersShephard1958}; see also \cite[Section~4.2]{ArtsteinSadovskySanyal}. Wang, Xue, and Zong \cite[Section~2]{WangXueZong} used the same body, denoted by $K^\bullet$, to transfer partition estimates to $o$-symmetric bodies in one higher dimension. Here we express the
volume of $(\susp K)^\circ$ in terms of $\vol_n(D(K)^\circ)$ and bound the volume of $\susp K$ in terms of $\vol_n(K)$. These relations allow us to apply the symmetric Mahler inequality to $\susp K$ and deduce Makai's inequality for $K$. No choice of center for $K$ is needed.

\begin{lemma}\label{lem:suspension}
For every convex body $K\subset\R^n$,
\begin{equation}\label{eq:suspension-pair}
 \begin{aligned}
 \vn{n+1}{(\susp K)^\circ}&=a_n\vn n{D(K)^\circ},\\
 \vn{n+1}{\susp K}&\leq a_n\vn nK.
 \end{aligned}
\end{equation}
Equality in the second relation holds if and only if $K$ is a simplex.
\end{lemma}

\begin{proof}
Choose affinely independent $v_0,\ldots,v_n\in K$. The vectors $(v_i,1)$ are linearly independent. The convex hull of these vectors and their negatives is a full-dimensional cross-polytope contained in $\susp K$. Hence $0\in\interior(\susp K)$.

The support function of $K$ is $h_K(y)=\max_{x\in K}\langle x,y\rangle$, $y\in\R^n$. By the definition of $\susp K$, we have
$$(\susp K)^\circ=\{(y,t)\in\R^n\times\R:h_K(y)+t\leq1,\ h_K(-y)-t\leq1\}.$$
Since $h_{D(K)}(y)=h_K(y)+h_K(-y)$, the projection of $(\susp K)^\circ$ onto $\R^n$ is
$$\{y\in\R^n:h_{D(K)}(y)\leq 2\}=2D(K)^\circ.$$
For each $y\in2D(K)^\circ$, the vertical section of $(\susp K)^\circ$ is the interval $[h_K(-y)-1,\,1-h_K(y)],$ whose length is $2-h_{D(K)}(y)$. Integrating the lengths of these vertical sections over $2D(K)^\circ$ and using the substitution $y=2z$ together with the homogeneity of $h_{D(K)}$, we obtain
\begin{align*}
 \vn{n+1}{(\susp K)^\circ}
 &=\int_{2D(K)^\circ}(2-h_{D(K)}(y))\,dy\\
 &=2^{n+1}\int_{D(K)^\circ}(1-h_{D(K)}(z))\,dz\\
 &=2^{n+1}\int_0^1\vol_n\{z\in D(K)^\circ:h_{D(K)}(z)\leq t\}\,dt\\
 &=2^{n+1}\int_0^1t^n\vol_n(D(K)^\circ)\,dt
 =a_n\vol_n(D(K)^\circ).
\end{align*}
Here the third equality uses $1-h_{D(K)}(z)=\int_{h_{D(K)}(z)}^1dt$, and the fourth uses $\{z:h_{D(K)}(z)\leq t\}=tD(K)^\circ$ for $0\leq t\leq1$.

For the second relation, consider the Rogers--Shephard body
$$W=\conv\bigl((\{0\}\times K/2)\cup(\{1\}\times(-K/2))\bigr).$$
Applying the volume estimate of Artstein-Avidan and Putterman \cite[Lemma~8]{ArtsteinPutterman} to $W$, with parameter $1/2$, gives
$$\vol_{n+1}(W)\leq\frac{\vol_n(K)}{n+1},$$
and equality implies that $K$ is a simplex. The converse follows from the computation in the proof of \cite[Theorem~4]{ArtsteinPutterman}. Define the affine map
$$\Phi:\R\times\R^n\longrightarrow\R^n\times\R,
 \qquad
 \Phi(\theta,z)=(2z,1-2\theta)=(2z,-2\theta)+(0,1).$$
Since affine maps preserve convex hulls,
\begin{align*}
 \Phi(W)
 &=\conv\!\left(
   \Phi(\{0\}\times K/2)\cup
   \Phi(\{1\}\times(-K/2))\right)\\
 &=\conv\!\left((K\times\{1\})\cup((-K)\times\{-1\})\right)\\
 &=\susp K.
\end{align*}
The linear part of $\Phi$ is the map $L(\theta,z)=(2z,-2\theta)$. In the standard coordinate bases,
$$[L]=
 \begin{pmatrix}
  0_{n\times1}&2I_n\\
  -2&0_{1\times n}
 \end{pmatrix},
 \qquad
 |\det L|=2\,|\det(2I_n)|=2^{n+1}.$$
Therefore,
 $$\vol_{n+1}(\susp K)=2^{n+1}\vol_{n+1}(W)\leq\frac{2^{n+1}}{n+1}\vol_n(K).$$
This proves the second relation in~\eqref{eq:suspension-pair} and its equality condition.
\end{proof}

Write $B_1^n=\{x\in\R^n:\sum_i|x_i|\leq1\}$ for the standard cross-polytope. If $K=\conv\{v_0,\ldots,v_n\}$ is a simplex, then $\susp K$ is the linear image of $B_1^{n+1}$ under the map whose columns are $(v_i,1)$. Translations of $K$ correspond to determinant-one shears: $(x,t)\mapsto(x+tu,t)$ sends $\susp K$ onto $\susp{K+u}$.

\begin{proof}[Proof of Theorem~\ref{thm:transfer}]
Applying the symmetric Mahler inequality in dimension $n+1$ to $\susp K$ and combining it with Lemma~\ref{lem:suspension},
we obtain
$$\frac{4^{n+1}}{(n+1)!}
 \leq\vn{n+1}{\susp K}\,\vn{n+1}{(\susp K)^\circ}
 \leq\frac{4^{n+1}}{(n+1)^2}\vn nK\,\vn n{D(K)^\circ}.$$
Dividing by $4^{n+1}/(n+1)^2$ proves~\eqref{eq:makai}. If equality holds, equality must hold in the second relation of \eqref{eq:suspension-pair}, so $K$ is a simplex.

If $K$ is a simplex, $\susp K$ is a linear image of $B_1^{n+1}$, whose volume product is $4^{n+1}/(n+1)!$. Since $(B_1^{n+1})^\circ=[-1,1]^{n+1}$ and the volume product is invariant under invertible linear maps, both relations in \eqref{eq:suspension-pair} are exact, and hence equality holds in \eqref{eq:makai}.
\end{proof}

\section{Successive minima}

Let $\Lambda=B\Z^n$ be a full-rank lattice, where $B$ is an invertible basis matrix, and put $\det(\Lambda)=|\det B|$. Its dual lattice is
$$\Lambda^*=\{u\in\R^n:\langle u,v\rangle\in\Z\text{ for all }v\in\Lambda\}=B^{-T}\Z^n,$$
where $B^{-T}=(B^{-1})^T$. For a convex body $K\subset\R^n$ with $0\in\interior K$, its $i$th successive minimum with respect to $\Lambda$ is
$$\lambda_i(K,\Lambda)=\inf\{t>0:\dim\spanop(tK\cap\Lambda)\geq i\},\qquad 1\leq i\leq n.$$
Thus $\lambda_i$ is the least dilation factor for which the dilated body contains $i$ linearly independent lattice vectors. We use the upper bound in Minkowski's second theorem: for every $o$-symmetric convex body $C$,
\begin{equation}\label{eq:minkowski}
 \frac{\vn nC}{\det(\Lambda)}\prod_{i=1}^n\lambda_i(C,\Lambda)\leq2^n.
\end{equation}
For further background on successive minima and Minkowski's second theorem, see Aliev and Henk \cite[Sections~1--2]{AlievHenk2023}.

The symmetric bound below is Mahler's polar successive-minima conjecture \cite{Mahler1974}; see also \cite[Conjecture~4.2]{AlievHenk2023}.\footnote{Mahler's 1976 addendum \cite{Mahler1976} used Guggenheimer's claimed sharp volume-product bound. Guggenheimer later acknowledged that his proof did not hold for $n>3$ \cite{Guggenheimer1978}.} For general convex bodies, Henk and Xue \cite[Introduction]{HenkXue} proposed the product bound as a strengthening of Makai Jr.'s conjectured bound involving $\lambda_1(\cs(K)^\circ,\Lambda^*)^n$: the product of the successive minima is no smaller than this power. Both product bounds follow from \eqref{eq:minkowski} and the corresponding volume-product inequalities.

\begin{corollary}[Polar successive-minima bounds]\label{cor:polar}
Let $\Lambda$ be a full-rank lattice in $\R^n$. For every convex body $K\subset\R^n$,
\begin{equation}\label{eq:central-polar}
 \frac{\vn nK}{\det(\Lambda)}\geq\frac{n+1}{n!}\prod_{i=1}^n\lambda_i(\cs(K)^\circ,\Lambda^*).
\end{equation}
For every $o$-symmetric convex body $C\subset\R^n$,
\begin{equation}\label{eq:symmetric-polar}
 \frac{\vn nC}{\det(\Lambda)}\geq\frac{2^n}{n!}\prod_{i=1}^n\lambda_i(C^\circ,\Lambda^*).
\end{equation}
Both constants are best possible.
\end{corollary}

\begin{proof}
Theorem~\ref{thm:transfer} and $\cs(K)^\circ=2D(K)^\circ$ give
$$\vn nK\,\vn n{\cs(K)^\circ}\geq\frac{2^n(n+1)}{n!}.$$
Minkowski's upper bound applied to $\cs(K)^\circ$ and $\Lambda^*$ gives
$$\vn n{\cs(K)^\circ}\prod_i\lambda_i(\cs(K)^\circ,\Lambda^*)\leq2^n\det(\Lambda^*)=\frac{2^n}{\det(\Lambda)}.$$
Combining these inequalities proves \eqref{eq:central-polar}. The same argument with \eqref{eq:sm} proves \eqref{eq:symmetric-polar}.

The cross-polytope $B_1^n$ has volume $2^n/n!$, and all successive minima of its polar $[-1,1]^n$ with respect to $\Z^n$ are one, so it attains the symmetric constant. For the general constant, consider the $n$-dimensional simplex
\begin{equation}\label{eq:standard-simplex}
 S_n=\conv\{e_1,\ldots,e_n,-\mathbf1\},\qquad \mathbf1=(1,\ldots,1),\qquad \Lambda=\Z^n,
\end{equation}
where $e_1,\ldots,e_n$ are the standard basis vectors. Its volume is $(n+1)/n!$. Fix a nonzero vector $u=(u_1,\ldots,u_n)\in\Z^n$. 
The values of the linear functional $x\mapsto\langle u,x\rangle$ at the vertices of $S_n$ are $ u_1,\ldots,u_n,-\sum_{i=1}^{n}u_i.$
These integers sum to zero and are not all zero. Their maximum is therefore at least $1$ and their minimum at most $-1$. Hence
$$h_{\cs(S_n)}(u)=\frac{h_{S_n}(u)+h_{S_n}(-u)}{2}\geq1.$$
Also $h_{\cs(S_n)}(e_i)=1$ for every $i$. Thus no dilation $t\cs(S_n)^\circ$ with $0<t<1$ contains a nonzero lattice point, whereas $\cs(S_n)^\circ$ contains the independent vectors $e_i$. All its successive minima are one.
\end{proof}

The comparison between the successive minima of $K^\circ$ and $\cs(K)^\circ$ gives a bound involving $K^\circ$ itself.

\begin{corollary}\label{cor:raw-polar}
Let $K\subset\R^n$ be a convex body with $0\in\interior K$, and let $\Lambda$ be a full-rank lattice. Then
\begin{equation}\label{eq:raw-polar}
 \frac{\vn nK}{\det(\Lambda)}\geq\frac{n+1}{n!}\prod_{i=1}^n\lambda_i(K^\circ,\Lambda^*).
\end{equation}
The constant is best possible.
\end{corollary}

\begin{proof}
Henk and Xue \cite[Proposition~2.1]{HenkXue} proved that
$$\lambda_i(K^\circ,\Lambda^*)\leq\lambda_i(\cs(K)^\circ,\Lambda^*)\qquad(1\leq i\leq n).$$
Together with \eqref{eq:central-polar}, this proves \eqref{eq:raw-polar}. The simplex $S_n$ attains the bound: the preceding vertex-value argument gives $h_{S_n}(u)\geq1$ for nonzero integral $u$, while $h_{S_n}(e_i)=1$, so all successive minima of $S_n^\circ$ are one.
\end{proof}

Since the successive minima are nondecreasing, \eqref{eq:raw-polar} implies
$$\vn nK\geq\frac{n+1}{n!}\det(\Lambda)\lambda_1(K^\circ,\Lambda^*)^n.$$
This is the first-successive-minimum form of the lower bound in Conjecture~I of \'Alvarez Paiva, Balacheff, and Tzanev \cite{AlvarezBalacheffTzanev}; see also \cite[Conjecture~4.4]{AlievHenk2023}.

To obtain a bound for the successive minima of $K$ itself, we use the general Mahler theorem~\cite[Theorem~1.1]{OpenAIGeneral}. The Santal\'o point $s(K)$ is the unique point of $\interior K$ minimizing $\vol_n((K-z)^\circ)$ over $z\in\interior K$. The general Mahler conjecture asserts that
\begin{equation}\label{eq:gm}
 \vn nK\,\vn n{(K-s(K))^\circ}\geq\frac{(n+1)^{n+1}}{(n!)^2}.
\end{equation}

\begin{corollary}[Successive minima at the Santal\'o point]\label{cor:santalo}
Assume \eqref{eq:gm} in dimension $n$. Let $K\subset\R^n$ be a convex body with $s(K)=0$, and let $\Lambda$ be a full-rank lattice. Then
\begin{equation}\label{eq:santalo-minima}
 \frac{\vn nK}{\det(\Lambda)}\prod_{i=1}^n\lambda_i(K,\Lambda)\geq\frac{n+1}{n!}.
\end{equation}
The constant is best possible.
\end{corollary}

\begin{proof}
For a convex body $L\subset\R^n$, denote its centroid by
$$g(L)=\frac{1}{\vol_n(L)}\int_Lx\,dx.$$
By the characterization of the Santal\'o point (see Meyer and Reisner~\cite{MeyerReisner}), the assumption $s(K)=0$ implies
$$g(K^\circ)=0.$$
We may therefore apply the volume estimate of Henk and Xue \cite[Theorem~1.2(ii)]{HenkXue} to $K^\circ$ and $\Lambda^*$, obtaining
$$\vol_n(K^\circ)\leq\frac{(n+1)^n}{n!\det(\Lambda)}\prod_{i=1}^{n}\lambda_i(K,\Lambda).$$
Combining this upper bound with~\eqref{eq:gm} yields \eqref{eq:santalo-minima}.

Since the Santal\'o point of a simplex coincides with its centroid, the simplex $S_n$ satisfies
$$s(S_n)=g(S_n)=\frac{1}{n+1}\left(\sum_{i=1}^{n}e_i-\mathbf1\right)=0.$$
As $S_n\subset[-1,1]^n$ and $e_1,\ldots,e_n\in S_n$, all its successive minima with respect to $\Z^n$ are one, and it attains the constant.
\end{proof}

\begin{remark}
Under the centroid condition $g(K)=0$, inequality~\eqref{eq:santalo-minima} was proved by Henk, Henze, and Hern\'andez Cifre
\cite[Theorem~1.2]{HenkHenzeHernandez}. Corollary~\ref{cor:santalo} establishes the same bound under the condition
$s(K)=0$, which is generally not equivalent to $g(K)=0$.
\end{remark}

Inserting \eqref{eq:sm} into Mahler's product estimate \cite[Section~3]{Mahler1939} gives the following consequence.

\begin{corollary}[A product bound for successive minima]\label{cor:transference}
For every $o$-symmetric convex body $C\subset\R^n$ and every full-rank lattice $\Lambda$,
$$\prod_{i=1}^n\lambda_i(C,\Lambda)\lambda_{n+1-i}(C^\circ,\Lambda^*)\leq n!.$$
\end{corollary}

\begin{proof}
Apply the upper bound in~\eqref{eq:minkowski} to $C$ with respect to $\Lambda$ and to $C^\circ$ with respect to $\Lambda^*$.
Multiplying the resulting inequalities and using the symmetric Mahler inequality~\eqref{eq:sm}, we obtain
\begin{align*}
 \prod_{i=1}^{n}
 \lambda_i(C,\Lambda)\lambda_{n+1-i}(C^\circ,\Lambda^*)
 &=\left(\prod_{i=1}^{n}\lambda_i(C,\Lambda)\right)
   \left(\prod_{i=1}^{n}\lambda_i(C^\circ,\Lambda^*)\right)\\
 &\leq\frac{4^n}{\vol_n(C)\vol_n(C^\circ)}\leq n!.
\end{align*}
\end{proof}

For an $o$-symmetric convex body $C\subset\R^n$, its critical determinant is
$$\Delta(C)=\inf\{\det(\Lambda):\Lambda\text{ is a full-rank lattice in }\R^n,\ \interior C\cap\Lambda=\{0\}\}.$$
Minkowski's first theorem and \eqref{eq:sm} give a corresponding product bound.

\begin{corollary}[Critical determinants]\label{cor:critical}
Every $o$-symmetric convex body $C\subset\R^n$ satisfies
$$\Delta(C)\Delta(C^\circ)\geq\frac{1}{n!}.$$
\end{corollary}

\begin{proof}
Minkowski's first theorem gives $\Delta(C)\geq 2^{-n}\vol_n(C)$. Applying this also to $C^\circ$ and using \eqref{eq:sm}, we obtain
$$\Delta(C)\Delta(C^\circ)\geq\frac{\vn nC\,\vn n{C^\circ}}{4^n}\geq\frac{1}{n!}.$$
\end{proof}

For $n=2$, this recovers Mahler's lower bound \cite[Theorem~5]{Mahler1948}, proved by the same argument.

\section{First covering minima and non-separability}

Let $K\subset\R^n$ be a convex body and $\Lambda$ a full-rank lattice. Following Kannan and Lov\'asz \cite{KannanLovasz}, define the first covering minimum by
$$\mu_1(K,\Lambda)=\inf\{t>0:(tK+\Lambda)\cap H\ne\varnothing\text{ for every affine hyperplane }H\}.$$
Thus $\mu_1$ is the infimum of the scales for which $tK+\Lambda$ meets every affine hyperplane. The lattice width of $K$ with respect to $\Lambda$ is
$$w_\Lambda(K)=\min_{u\in\Lambda^*\setminus\{0\}}(h_K(u)+h_K(-u)).$$
The standard duality relation \cite[Section~1, formula (1.5)]{HenkSchymuraXue} reads
\begin{equation}\label{eq:width}
 \mu_1(K,\Lambda)=\frac{1}{w_\Lambda(K)}=\frac{1}{2\lambda_1(\cs(K)^\circ,\Lambda^*)}.
\end{equation}

\begin{corollary}\label{cor:covering}
For every convex body $K\subset\R^n$ and every full-rank lattice $\Lambda$,
\begin{equation}\label{eq:covering}
 \frac{\vn nK}{\det(\Lambda)}\mu_1(K,\Lambda)^n\geq\frac{n+1}{2^n n!}.
\end{equation}
For every $o$-symmetric convex body $C\subset\R^n$,
$$\frac{\vn nC}{\det(\Lambda)}\mu_1(C,\Lambda)^n\geq\frac{1}{n!}.$$
Both constants are best possible.
\end{corollary}

\begin{proof}
Since $\prod_i\lambda_i(\cs(K)^\circ,\Lambda^*)\geq\lambda_1(\cs(K)^\circ,\Lambda^*)^n$, the general bound follows
from \eqref{eq:central-polar} and \eqref{eq:width}. For an $o$-symmetric convex body $C$, we have $\cs(C)=C$. Combining~\eqref{eq:symmetric-polar} with~\eqref{eq:width} in the same way proves the symmetric case. The simplex $S_n$ and the cross-polytope $B_1^n$ satisfy
$$\mu_1(S_n,\Z^n)=\mu_1(B_1^n,\Z^n)=\frac{1}{2}$$
and attain the two constants, respectively.
\end{proof}

Makai Jr.~\cite[Theorem~1]{Makai} established a connection between non-separable lattice families and lattice packings through polarity. The conjectured sharp density bounds for these families are stated in Makai and Martini~\cite[Conjecture~3.14]{MakaiMartini}. We now derive these bounds from Corollary~\ref{cor:covering}.

The lattice family $\{K+v:v\in\Lambda\}$ is called non-separable if every affine hyperplane meets at least one translate. Its density, counting multiplicity, is $\vol_n(K)/\det(\Lambda)$. The lattice-width characterization underlying \eqref{eq:width} says that this family is non-separable precisely when $w_\Lambda(K)\geq 1$, or equivalently $\mu_1(K,\Lambda)\leq1$. Corollary~\ref{cor:covering} therefore gives
$$\frac{\vn nK}{\det(\Lambda)}\geq\frac{n+1}{2^n n!},\qquad \frac{\vn nC}{\det(\Lambda)}\geq\frac{1}{n!}$$
for non-separable lattice families generated by a convex body $K$ and by an $o$-symmetric convex body $C$, respectively. Equality holds for $K=S_n/2$ and $C=B_1^n/2$ with $\Lambda=\Z^n$. In dimension three, the constants are $1/12$ and $1/6$. These bounds and equality examples establish both assertions of~\cite[Conjecture~3.14]{MakaiMartini}. 

\section{Mixed-volume consequences}

The mixed volume $V(K_1,\ldots,K_n)$ is defined by requiring $n!V(K_1,\ldots,K_n)$ to be the coefficient of $t_1\cdots t_n$ in the volume polynomial $\vol_n(t_1K_1+\cdots+t_nK_n)$, $t_i\geq0$. We write $K[j]$ for $j$ copies of $K$; in particular, $V(K[n])=\vol_n(K)$. Repeated application of the Alexandrov--Fenchel inequality gives
\begin{equation}\label{eq:af}
 V(K_1,\ldots,K_n)
 \geq\left(\prod_{i=1}^{n}\vol_n(K_i)\right)^{1/n}
\end{equation}
for full-dimensional convex bodies $K_1,\ldots,K_n\subset\R^n$; see~\cite[Sections~7.3--7.4]{Schneider}. Bourgain and Milman~\cite{BourgainMilman} combined this inequality
with their reverse Santal\'o inequality to obtain a lower bound for
 $$V(K_1,\ldots,K_n)V(K_1^\circ,\ldots,K_n^\circ),
 \qquad 0\in\interior K_i\ (1\leq i\leq n).$$
Applying the same argument to the Mahler and Makai inequalities yields the following sharp bounds. Repeated bodies are allowed.

\begin{corollary}[Paired mixed volumes]\label{cor:mixed}
\leavevmode
\begin{enumerate}
\item For $o$-symmetric convex bodies $C_1,\ldots,C_n\subset\R^n$,
$$V(C_1,\ldots,C_n)V(C_1^\circ,\ldots,C_n^\circ)\geq\frac{4^n}{n!}.$$
\item For convex bodies $K_1,\ldots,K_n\subset\R^n$,
$$V(K_1,\ldots,K_n)V((K_1-s(K_1))^\circ,\ldots,(K_n-s(K_n))^\circ)\geq\frac{(n+1)^{n+1}}{(n!)^2}.$$
\item For convex bodies $K_1,\ldots,K_n\subset\R^n$,
$$V(K_1,\ldots,K_n)V(D(K_1)^\circ,\ldots,D(K_n)^\circ)\geq\frac{n+1}{n!}.$$
\end{enumerate}
Each constant is best possible.
\end{corollary}

\begin{proof}
For two lists of convex bodies $K_1,\ldots,K_n$ and $L_1,\ldots,L_n$, applying \eqref{eq:af} to each list gives
$$V(K_1,\ldots,K_n)V(L_1,\ldots,L_n)\geq\left(\prod_{i=1}^n\bigl(\vn n{K_i}\,\vn n{L_i}\bigr)\right)^{1/n}.$$
The volume-product inequalities \eqref{eq:sm}, \eqref{eq:gm}, and \eqref{eq:makai} give the three conclusions. For sharpness, take all $C_i$ equal to the same cube in part~(1), and all $K_i$ equal to the same simplex in parts~(2) and~(3).
\end{proof}

Taking all bodies equal recovers the corresponding volume-product bound, so each assertion is equivalent to it in the presence of Alexandrov--Fenchel.

\begin{remark}[A reflected mixed-volume form]\label{rem:reflected}

Integrating the horizontal sections of $\susp K$, as in the proof
of~\cite[Lemma~4.5]{ArtsteinSadovskySanyal}, gives
\begin{align*}
 \vn{n+1}{\susp K}
 &=2\int_0^1\vol_n(tK+(1-t)(-K))\,dt\\
 &=\frac{2}{n+1}\sum_{j=0}^n V(K[j],-K[n-j]).
\end{align*}
Combining the preceding formula for $\vol_{n+1}(\susp K)$ with
the polar-volume identity in~\eqref{eq:suspension-pair} and applying
the symmetric Mahler inequality in dimension $n+1$, we obtain
\begin{align*}
 \frac{2^{n+2}}{(n+1)^2}\vol_n(D(K)^\circ)
 \sum_{j=0}^{n}V(K[j],-K[n-j])
 &=\vol_{n+1}(\susp K)\vol_{n+1}((\susp K)^\circ)\\
 &\geq\frac{4^{n+1}}{(n+1)!}.
\end{align*}
Hence
\begin{equation}\label{eq:reflected}
 \vol_n(D(K)^\circ)\sum_{j=0}^{n}V(K[j],-K[n-j])
 \geq\frac{2^n(n+1)}{n!}.
\end{equation}
Equality holds when $K$ is a simplex, since $\susp K$ is then a linear image of a cross-polytope, as observed after Lemma~\ref{lem:suspension}.

Moreover, the volume upper bound in~\eqref{eq:suspension-pair} can be rewritten as
$$\sum_{j=0}^{n}V(K[j],-K[n-j])\leq2^n\vol_n(K).$$
Together with~\eqref{eq:reflected}, this gives
\begin{align*}
 \vol_n(K)\vol_n(D(K)^\circ)
 &\geq\frac{1}{2^n}\vol_n(D(K)^\circ)
       \sum_{j=0}^{n}V(K[j],-K[n-j])\\
 &\geq\frac{n+1}{n!},
\end{align*}
which is Makai's inequality. Thus~\eqref{eq:reflected} provides a mixed-volume formulation of the proof of Theorem~\ref{thm:transfer}.
\end{remark}


\begin{thebibliography}{99}

\bibitem{Aliev}
A. Aliev, Mahler-type volume inequality for convex bodies with tetrahedral symmetry, preprint (2025), \href{https://arxiv.org/abs/2511.14991v1}{arXiv:2511.14991v1}.

\bibitem{AlievHenk2023}
I. Aliev and M. Henk, Minkowski's successive minima in convex and discrete geometry, \emph{Commun. Math.} \textbf{31} (2023), 35--59, DOI: \href{https://doi.org/10.46298/cm.11155}{10.46298/cm.11155}.

\bibitem{AlvarezBalacheffTzanev}
J. C. \'Alvarez Paiva, F. Balacheff and K. Tzanev, Isosystolic inequalities for optical hypersurfaces, \emph{Adv. Math.} \textbf{301} (2016), 934--972, DOI: \href{https://doi.org/10.1016/j.aim.2016.07.003}{10.1016/j.aim.2016.07.003}.

\bibitem{ArtsteinPutterman}
S. Artstein-Avidan and E. Putterman, On unbalanced difference bodies and Godbersen's conjecture, \emph{Proc. Amer. Math. Soc.} \textbf{153} (2025), 5345--5359, DOI: \href{https://doi.org/10.1090/proc/17428}{10.1090/proc/17428}.

\bibitem{ArtsteinSadovskySanyal}
S. Artstein-Avidan, S. Sadovsky and R. Sanyal, Geometric inequalities for anti-blocking bodies, \emph{Commun. Contemp. Math.} \textbf{25} (2023), 2150113, DOI: \href{https://doi.org/10.1142/S0219199721501133}{10.1142/S0219199721501133}.

\bibitem{BourgainMilman}
J. Bourgain and V. D. Milman, New volume ratio properties for convex symmetric bodies in $\R^n$, \emph{Invent. Math.} \textbf{88} (1987), 319--340, DOI: \href{https://doi.org/10.1007/BF01388911}{10.1007/BF01388911}.

\bibitem{Guggenheimer1978}
H. Guggenheimer, Correction to ``Polar reciprocal convex bodies'', \emph{Israel J. Math.} \textbf{29} (1978), 312, DOI: \href{https://doi.org/10.1007/BF02762019}{10.1007/BF02762019}.

\bibitem{HenkHenzeHernandez}
M. Henk, M. Henze and M. A. Hern\'andez Cifre, Variations of Minkowski's theorem on successive minima, \emph{Forum Math.} \textbf{28} (2016), 311--325, DOI: \href{https://doi.org/10.1515/forum-2014-0093}{10.1515/forum-2014-0093}.

\bibitem{HenkSchymuraXue}
M. Henk, M. Schymura and F. Xue, Packing minima and lattice points in convex bodies, \emph{Mosc. J. Comb. Number Theory} \textbf{10} (2021), 25--48, DOI: \href{https://doi.org/10.2140/moscow.2021.10.25}{10.2140/moscow.2021.10.25}.

\bibitem{HenkXue}
M. Henk and F. Xue, On successive minima-type inequalities for the polar of a convex body, \emph{Rev. R. Acad. Cienc. Exactas F\'is. Nat. Ser. A Mat. RACSAM} \textbf{113} (2019), 2601--2616, DOI: \href{https://doi.org/10.1007/s13398-019-00648-x}{10.1007/s13398-019-00648-x}.

\bibitem{KannanLovasz}
R. Kannan and L. Lov\'asz, Covering minima and lattice-point-free convex bodies, \emph{Ann. of Math. (2)} \textbf{128} (1988), 577--602, DOI: \href{https://doi.org/10.2307/1971436}{10.2307/1971436}.

\bibitem{Mahler1939}
K. Mahler, Ein \"Ubertragungsprinzip f\"ur konvexe K\"orper, \emph{\v{C}asopis P\v{e}st. Mat. Fys.} \textbf{68} (1939), 93--102, DOI: \href{https://doi.org/10.21136/CPMF.1939.109441}{10.21136/CPMF.1939.109441}.

\bibitem{Mahler1948}
K. Mahler, On lattice points in polar reciprocal convex domains, \emph{Indag. Math.} \textbf{10} (1948), 176--179.

\bibitem{Mahler1974}
K. Mahler, Polar analogues of two theorems by Minkowski, \emph{Bull. Austral. Math. Soc.} \textbf{11} (1974), 121--129, DOI: \href{https://doi.org/10.1017/S0004972700043690}{10.1017/S0004972700043690}.

\bibitem{Mahler1976}
K. Mahler, An addition to a note of mine, \emph{Bull. Austral. Math. Soc.} \textbf{14} (1976), 397--398, DOI: \href{https://doi.org/10.1017/S0004972700025296}{10.1017/S0004972700025296}.

\bibitem{Makai}
E. Makai, Jr., On the thinnest non-separable lattice of convex bodies, \emph{Studia Sci. Math. Hungar.} \textbf{13} (1978), 19--27.

\bibitem{MakaiMartini}
E. Makai, Jr., and H. Martini, Density estimates for $k$-impassable lattices of balls and general convex bodies in $\R^n$, preprint (2016), \href{https://arxiv.org/abs/1612.01307v1}{arXiv:1612.01307v1}.

\bibitem{MeyerReisner}
M. Meyer and S. Reisner, The convex intersection body of a convex body, \emph{Glasg. Math. J.} \textbf{53} (2011), 523--534, DOI: \href{https://doi.org/10.1017/S0017089511000103}{10.1017/S0017089511000103}.

\bibitem{OpenAIGeneral}
OpenAI, The Mahler conjecture for general convex bodies, preprint (September 22, 2026), repository version \href{https://github.com/openai/math/blob/adc7f1241b42e322a6451854ab7e4b4c146bf78a/preprints/The-Mahler-Conjecture-for-General-Convex-Bodies-September-22-2026/paper.pdf}{\texttt{adc7f1241b42}}.

\bibitem{OpenAISymmetric}
OpenAI, The symmetric Mahler conjecture and its equality cases, preprint (September 22, 2026), repository version \href{https://github.com/openai/math/blob/adc7f1241b42e322a6451854ab7e4b4c146bf78a/preprints/The-symmetric-Mahler-conjecture-and-its-equality-cases-September-22-2026/paper.pdf}{\texttt{adc7f1241b42}}.

\bibitem{RogersShephard1958}
C. A. Rogers and G. C. Shephard, Convex bodies associated with a given convex body, \emph{J. London Math. Soc.} \textbf{33} (1958), 270--281, DOI: \href{https://doi.org/10.1112/jlms/s1-33.3.270}{10.1112/jlms/s1-33.3.270}.

\bibitem{Schneider}
R. Schneider, \emph{Convex bodies: the Brunn--Minkowski theory}, 2nd expanded ed., Encyclopedia Math. Appl. 151, Cambridge University Press, Cambridge, 2014, DOI: \href{https://doi.org/10.1017/CBO9781139003858}{10.1017/CBO9781139003858}.

\bibitem{WangXueZong}
J. Wang, F. Xue and C. Zong, On Boltyanski and Gohberg's partition conjecture, \emph{Bull. London Math. Soc.} \textbf{56} (2024), 140--149, DOI: \href{https://doi.org/10.1112/blms.12919}{10.1112/blms.12919}.

\end{thebibliography}
\end{document}